\documentclass[11pt,a4paper]{amsart}

\usepackage[latin1]{inputenc}
\usepackage{amsmath,amsfonts,amssymb,setspace,amscd}
\usepackage{enumerate}
\usepackage[
      colorlinks=true,    
      urlcolor=blue,    
      menucolor=blue,    
      linkcolor=blue,    
      bookmarks=true,    
      citecolor=blue,
      bookmarksopen=true,    
      hyperfootnotes=false,    
      pdfpagemode=UseOutlines    
]{hyperref}
\usepackage[margin=2.5cm]{geometry}
\onehalfspace

\usepackage{marginnote}

\newtheorem{theorem}{Theorem}[section]
\newtheorem{lemma}[theorem]{Lemma}
\newtheorem{proposition}[theorem]{Proposition}

\newtheorem{corollary}[theorem]{Corollary}

\theoremstyle{definition}
\newtheorem{definition}[theorem]{Definition}
\newtheorem*{remark}{Remark}

\numberwithin{equation}{section}
\numberwithin{theorem}{section}

\newcommand{\R}{\mathbb{R}}
\newcommand{\Z}{\mathbb{Z}}
\newcommand{\Q}{\mathbb{Q}}
\newcommand{\C}{\mathbb{C}}
\newcommand{\Gm}{\mathbb{G}_\textrm{m}}

\newcommand{\Qbar}{\overline{\Q}}

\newcommand{\cC}{\mathcal{C}}
\newcommand{\cE}{\mathcal{E}}
\newcommand{\cA}{\mathcal{A}}

\newcommand{\bo}[1]{\boldsymbol{#1}}
\newcommand{\mc}[1]{\mathcal{#1}}

\def\lg{\left\lbrace}
\def\rg{\right\rbrace}

\author[F. Barroero]{Fabrizio Barroero}
\address{Department of Mathematics and Computer Science, University of Basel, Spiegelgasse 1, 4051 Basel, Switzerland}
\email{fbarroero@gmail.com}

\author[L. Capuano]{Laura Capuano}
\address{Mathematical Institute,
University of Oxford,
Woodstock Road
Oxford
OX2 6GG, UK}
\email{laura.capuano1987@gmail.com}

\title[Unlikely intersections]{Unlikely intersections in products of families of elliptic curves and the multiplicative group}
\date{\today}

\begin{document}

\begin{abstract}
Let $E_\lambda$ be the Legendre elliptic curve of equation $Y^2=X(X-1)(X-\lambda)$. We recently proved that, given $n$ linearly independent points $P_1(\lambda), \dots,P_n(\lambda)$ on $E_\lambda$ with coordinates in $\overline{\mathbb{Q}(\lambda)}$, there are at most finitely many complex numbers $\lambda_0$ such that the points $P_1(\lambda_0), \dots,P_n(\lambda_0)$ satisfy two independent relations on $E_{\lambda_0}$. In this article we continue our investigations on Unlikely Intersections in families of abelian varieties and consider the case of a curve in a product of two non-isogenous families of elliptic curves and in a family of split semi-abelian varieties.
\end{abstract}

\maketitle

\section{Introduction}

Let $n, m$ be positive integers and let $E_\lambda$ denote the elliptic curve with Legendre equation
\begin{equation}\label{legendre}
Y^2=X(X-1)(X-\lambda).
\end{equation}

We consider an irreducible curve $\cC\subseteq \mathbb{A}^{2n+2m+2}$, defined over $\Qbar$, with coordinate functions $$(x_1,y_1,\dots ,x_{n}, y_{n}, \lambda,u_1,v_1, \dots , u_{m},v_{m},\mu ),$$ such that, for every $i=1,\dots , n $, the points $P_i=(x_i,y_i)$ lie on the elliptic curve $E_\lambda$ and, for every $j=1,\dots , m $, the points $Q_j=(u_j,v_j)$ lie on $E_\mu$. We will assume that $\lambda,\mu \neq 0,1 $ on the curve. Therefore, $\cC $ is not Zariski closed in $\mathbb{A}^{2n+2m+2}$ but it is locally closed.

We call $R_1$ and $R_2$ the endomorphism rings of $E_\lambda$ and $E_\mu$, respectively. These will be isomorphic to $\Z$, unless we have a fixed elliptic curve with complex multiplication. For instance, if $\mu=\mu_0$ is constant on $\cC$ and $E_{\mu_0}$ has complex multiplication, then $R_2$ will be strictly larger than $\Z$.

Suppose that, on $\cC$, the two elliptic curves $E_\lambda$ and $E_\mu$ are not isogenous (for instance, we must have $\lambda\neq \mu^{\pm 1}$) and that the $P_i$ and $Q_j$ are independent, i.e., there is no non-trivial relation among them over $R_1$ or $R_2$, respectively.

Now, as $\bo{c}$ varies on $\cC(\C)$, the specialized points $P_i(\bo{c})$ and $Q_j(\bo{c})$ will be lying on the specialized elliptic curves $E_{\lambda(\bo{c})}$ and $E_{\mu(\bo{c})}$, respectively. We implicitly exclude the finitely many $\bo{c}$ with $\lambda(\bo{c})$ or $\mu(\bo{c})$ equal to 0 or 1, since in that case we have a singular curve.

It might happen that, for a certain $\bo{c}$, the specialized points become dependent over $R_1$ or $R_2$, or an eventually larger endomorphism ring. We do not consider the latter case and we will talk about relations among the generic and specialized points always meaning relations over $R_1$ and $R_2$.

In \cite{linrel} we proved that, in case $\lambda$ is non-constant and the $P_i$ are independent on $\cC$, there are at most finitely many $\bo{c} \in \cC(\C)$ such that $P_1(\bo{c}), \dots , P_n(\bo{c})$ satisfy two independent relations on $E_{\lambda(\bo{c})}$ (see \cite{MZ12} for the case $n=2$).

In this article we continue our program of studying Unlikely Intersections in families of abelian varieties and prove the following theorem.

\begin{theorem} \label{mainthm}
Let $\cC\subseteq \mathbb{A}^{2n+2m+2}$ be an irreducible curve defined over $\Qbar$ with coordinate functions $(x_1,y_1,\dots ,x_n, y_n, \lambda, u_1,v_1, \dots , u_{m},v_{m},\mu)$, such that, for every $i=1,\dots , n $, the points $P_i=(x_i,y_i)$ lie on $E_\lambda$ and, for every $j=1,\dots , m $, the $Q_j=(u_j,v_j)$ lie on $E_\mu$. Suppose moreover that $E_\lambda$ and $E_\mu$ are not isogenous and that there are no generic non-trivial relations among $P_1,\dots , P_n$ on $E_\lambda$ and among $Q_1,\dots , Q_m$ on $E_\mu$.
Then, there are at most finitely many $\bo{c}\in \cC(\C)$ such that there exist $(a_1,\dots, a_n) \in R_1^n\setminus \{0 \}$ and $( b_1, \dots , b_m) \in R_2^m\setminus \{0 \}$ for which
\begin{equation*}\label{rel1}
a_1P_1(\bo{c})+\dots +a_nP_n(\bo{c})= O \text{  and     } \  b_1Q_1(\bo{c})+\dots +b_m Q_m(\bo{c})= O .
\end{equation*}
\end{theorem}

In case $n=m=1$, the theorem says that there are at most finitely many points on the curve $\cC$ such that $P_1$ and $Q_1$ are simultaneously of finite order on the respective specialized elliptic curves. This is nothing but the Proposition on p. 120 of \cite{MZ14a}. Actually, Masser and Zannier deal also with the case of a curve $\cC$ not defined over the algebraic numbers. Note that, if $\lambda$ and $\mu$ are both constant on $\cC$ and $n=m=1$, then the conclusion of the theorem is a special case of Raynaud's Theorem \cite{Rayn}, also known as the Manin-Mumford Conjecture.

For general $n$ and $m$, in the case of two constant elliptic curves defined over the algebraic numbers, the theorem follows from the recent work \cite{HabPila14} of Habegger and Pila. Therefore, we can suppose that at least one of the the two parameters, say $\lambda$, is non-constant and that $R_1\cong \Z$.\\

We also obtain a similar result for the fibered product of $n$ copies of $E_\lambda$ with $\Gm^m=(\C^\times)^m$.
We consider a curve $\cC\subseteq \mathbb{A}^{2n+1}\times \Gm^m$ with coordinate functions $$(x_1,y_1,\dots ,x_{n}, y_{n}, \lambda,u_1, \dots , u_{m} ),$$ with $\lambda$ non-constant, such that, for every $i=1,\dots , n $, the points $P_i=(x_i,y_i)$ lie on $E_\lambda$ as above. As the point $\bo{c}$ varies on the curve $\cC$, the $u_j (\bo{c})$ will be non-zero complex numbers.

\begin{theorem} \label{mainthm2}
Let $\cC\subseteq \mathbb{A}^{2n+1}\times \mathbb{G}_\emph{m}^m$ be an irreducible curve defined over $\Qbar$ with coordinate functions $(x_1,y_1,\dots ,x_{n}, y_{n}, \lambda,u_1, \dots , u_{m} )$, $\lambda$ non-constant, such that, for every $i=1,\dots , n $, the points $P_i=(x_i,y_i)$ lie on $E_\lambda$. Suppose moreover that no generic non-trivial relation among $P_1,\dots , P_n$ holds and that the $u_1, \dots , u_{m}$ are generically multiplicatively independent.
Then, there are at most finitely many $\bo{c}\in \cC(\C)$ such that there exist $(a_1,\dots, a_n) \in \Z^n\setminus \{0 \}$ and $( b_1, \dots , b_m) \in \Z^m \setminus \{0 \}$ for which
\begin{equation*}\label{rel2}
a_1P_1(\bo{c})+\dots +a_nP_n(\bo{c})= O \text{  and     } u_1(\bo{c})^{b_1}\cdots u_m(\bo{c})^{b_m}=1 .
\end{equation*}
\end{theorem}

Here, the case $n=m=1$ ($P_1$ torsion and $u_1$ a root of 1) follows from work of Bertrand, Masser, Pillay and Zannier \cite{BMPZ}.
In some special cases, Habegger, Jones and Masser \cite{HJM} recently gave an effective (but not explicit) bound for the degree of the set of ``special'' points, while in some more specific cases Stoll \cite{Stoll} proved emptiness, e.g., there is no root of unity $\lambda_0\neq 1$ such that $\left(2,\sqrt{2(2-\lambda_0)}\right)$ is torsion on $E_{\lambda_0}$.\\

Let us see a few examples. Consider the points 
\begin{align*}
P_1(\lambda)=\left(2,\sqrt{2(2-\lambda)} \right), \qquad P_2(\lambda)=\left(3,\sqrt{6(3-\lambda)} \right) ,
\end{align*}
on $E_\lambda$ and 
\begin{align*}
Q_1(\lambda)=\left(2,\sqrt{2(2+\lambda)} \right), \qquad Q_2(\lambda)=\left(3,\sqrt{6(3+\lambda)} \right) ,
\end{align*}
on $E_{-\lambda}$. The two elliptic curves $E_\lambda$ and $E_{-\lambda}$ are not identically isogenous. In fact, if they were, each $j$-invariant would be integral over the ring generated by the other over $\C$ and it is easy to prove that this is not the case (see Section 12 of \cite{MZ14a}). Moreover, $P_1$ and $P_2$ are not identically dependent on $E_{\lambda}$. Indeed, since these two points are defined over disjoint quadratic extensions of $\Qbar(\lambda)$, by conjugating one can see that the existence of a relation would imply that the points are identically of finite order on $E_\lambda$ and this is not the case (see p.68 of \cite{Zannier}). For the same reason $Q_1$ and $Q_2$ are not identically dependent on $E_{-\lambda}$. Theorem \ref{mainthm} then implies that there are at most finitely many complex $\lambda_0$ such that there are $(a_1,a_2),(b_1,b_2)\in \Z^2 \setminus \{0\}$ with $a_1P_1(\lambda_0)+a_2 P_2(\lambda_0)= O$ on $E_{\lambda_0}$ and $b_1Q_1(\lambda_0)+b_2 Q_2(\lambda_0)= O$ on $E_{-\lambda_0}$.

Now, consider $E_{-1}$. This is an elliptic curve with complex multiplication by the gaussian integers $\Z[i]$. Let $P_1(\lambda)$ and $P_2(\lambda)$ be as in the example above and let 
\begin{align*}
Q_1(\lambda)=\left(\lambda,\sqrt{\lambda(\lambda-1)(\lambda+1)} \right), \qquad Q_2(\lambda)=\left(2\lambda,\sqrt{2\lambda(2\lambda-1)(2\lambda+1)} \right),
\end{align*}
on $E_{-1}$. The two points $Q_1$ and $Q_2$ are not identically dependent on $E_{-1}$. Indeed, they are defined over disjoint quadratic extensions and they are not identically torsion.
Therefore, Theorem \ref{mainthm} implies that there are at most finitely many complex $\lambda_0$ such that there are $(a_1,a_2)\in \Z^2 \setminus \{0\}$ and $(b_1,b_2)\in \Z[i]^2\setminus\{0\}$ with $a_1P_1(\lambda_0)+a_2 P_2(\lambda_0)= O$ on $E_{\lambda_0}$ and $b_1Q_1(\lambda_0)+b_2 Q_2(\lambda_0)= O$ on $E_{-1}$.

Finally, let $P_1$ and $P_2$ be as above. Then, Theorem \ref{mainthm2} implies that there are at most finitely many complex $\lambda_0$ such that there are $(a_1,a_2),(b_1,b_2)\in \Z^2 \setminus \{0\}$ with $a_1P_1(\lambda_0)+a_2 P_2(\lambda_0)= O$ on $E_{\lambda_0}$ and $\lambda_0^{b_1} (\lambda_0-1)^{b_2}=1$.\\

In general, there are infinitely many $\bo{c}_0$ such that $P_1(\bo{c}_0), \dots , P_n(\bo{c}_0)$ are dependent on $E_{\lambda(\bo{c}_0)}$. For instance, any $P_i$ specializes to a torsion point for infinitely many $\bo{c}_0$, see \cite{Zannier}, p.~92. On the other hand, a well-known theorem of Silverman \cite{Sil83} implies that the absolute Weil height of such points is bounded. A direct effective proof of this can be found in Masser's Appendix C of \cite{Zannier}. In particular, there are at most finitely many $\bo{c}_0$ yielding one relation and defined over a given number field or of bounded degree over $\Q$.

The proof of our Theorems follows the general strategy introduced by Pila and Zannier in \cite{PilaZannier} and used by Masser and Zannier in various articles \cite{MasserZannier08}, \cite{MasserZannier10}, \cite{MZ12} and \cite{MZ14a} and by the authors in \cite{linrel}. In particular, we consider the elliptic logarithms $z_1, \dots , z_n$ of $P_1 , \dots , P_n$ and $w_1, \dots ,w_m$ of $Q_1, \dots ,Q_m$ (or the principal determination of the standard logarithms of $u_1, \dots , u_m$ in the $\Gm$ case) and the equations
$$
z_i=p_if+q_ig,  \quad   w_j=r_j h + s_j k,
$$
for $i=1,\dots, n$ and $j=1,\dots , m$, where $f$ and $g$ are suitably chosen basis elements of the period lattice of $E_\lambda$ and $h$ and $k$ basis elements for the period lattice of $E_\mu$ (or $h=1$ and $k=2\pi i $ for $\Gm$). If we consider the real coordinates $p_i, q_i, r_j, s_j$ as functions of a local uniformizer on a compact disc $D$, the image of these functions in $\R^{2n+2m}$ is a subanalytic surface $S$. The points of $\cC$ that yield two relations will correspond to points of $S$ lying on linear varieties defined by equations of some special form and with integer coefficients.
Now, we use a recent result of Habegger and Pila \cite{HabPila14} building on an earlier work of Pila \cite{Pila11}, which in turn is a refinement of the Pila-Wilkie Theorem \cite{PilaWilkie}, to obtain an upper bound of order $T^\epsilon$ for the number of points of $S$ lying on subspaces of the special form mentioned above and rational coefficients of height at most $T$, provided the $z_i$ and the $w_j$ are algebraically independent. This is ensured by a result of Bertrand \cite{Ber09}, in case our curve $\cC$ is not contained in a translate of a proper algebraic subgroups by a constant point. This is always the case in the setting of Theorem \ref{mainthm} if both $\lambda$ and $\mu$ are non-constant. On the other hand, if $\mu=\mu_0$ is constant or we are in the setting of Theorem \ref{mainthm2}, our curve might be contained in a non-torsion translate of a proper algebraic subgroup (e.g., we might have 
$Q_1 \in E_{\mu_0}(\Qbar)$ of infinite order). In this case, we are able to prove the same estimate essentially by reducing to the case $m=1$.

Now, to conclude the proof, we use works of Masser \cite{Masser88}, \cite{Masser89} and David \cite{David97} and exploit the boundedness of the height to show that the number of points of $S$ considered above is of order at least $T^\delta$ for some $\delta>0$. Comparing the two estimates leads to an upper bound for $T$ and thus for the coefficients of the two relations, concluding the proof.

Our Theorem \ref{mainthm2} does not deal with the case of $\lambda$ constant on $\cC$ since
Silverman's bounded height Theorem requires $\lambda$ not to be constant. On the other hand, a result of Bombieri, Masser and Zannier \cite{BMZ99} gives boundedness of the height in case the $u_j$ are independent modulo constants, while Viada \cite{Viada2003} proved the analogous result for a constant elliptic curve $E$ defined over the algebraic numbers. Therefore, our proof goes through in the constant case, unless $(P_1,\dots , P_n)$ and $(u_1, \dots, u_m)$ are both contained in a non-torsion translate of an algebraic subgroup of $E^n$ and $\Gm^m$, respectively. \\

We now formulate a statement in scheme theoretic terms and in the flavor of the so-called Zilber-Pink conjectures.
Let $S$ be an irreducible non-singular quasi-projective curve defined over a number field $k$. Fix non-negative integers $l,p,q$, and positive integers $n_1,\dots, n_l, m_1, \dots ,m_p$. For $i=1,\dots , l$, let $\mathcal{E}_i \rightarrow S$ be non-isotrivial elliptic schemes such that the generic fibers are pairwise non-isogenous. By non-isotrivial we mean that it cannot become a constant family after a finite \'etale base change.
Now, for $i=1, \dots , l$ we let $\mc{A}_i$ be the $n_i$-fold fibered power of $\cE_i$ over $S$.
Let $E_1, \dots ,E_p$ be elliptic curves defined over $k$ which are pairwise non-isogenous. We consider these and the multiplicative group $\mathbb{G}_m^q$ as constant families over $S$, i.e., we call $E_j$ and $\mathbb{G}_m^q$ the fibered products $E_j\times_k S$ and $\mathbb{G}_m^q\times_k S$ respectively.
Finally we let $\mc{A}$ be the fibered product
$$
\mc{A}_1\times_S \dots \times_S \mc{A}_l \times_S E_1^{m_1}  \times_S\dots \times_S E_{p}^{m_p}\times_S \mathbb{G}_m^q
$$
 over $S$. This is a semiabelian scheme over $S$. We call $\pi$ the structure morphism $\cA \rightarrow S$.

A subgroup scheme $G$ of $\mc{A}$ is a closed subvariety, possibly reducible, which contains the image of the zero section $S\rightarrow \mc{A}$, is mapped to itself by the inversion morphism and such that the image of $G\times_S G$ under the addition morphism is in $G$.
A subgroup scheme $G$ is called flat if $\pi_{|_G} : G \rightarrow S$ is flat, i.e., all irreducible components of $G$ dominate the base curve $S$ (see \cite{Hart}, Proposition III 9.7).

We can now state the following theorem, which is a very special case of a conjecture of Pink \cite{Pink}, Conjecture 6.1.

\begin{theorem}\label{thmscheme}
Let $\mc{A}$ be as above and suppose that either $p$ or $q$ equal 0. Let $\mc{A}^{\{2\}}$ be the union of its flat subgroup schemes of codimension at least 2. Let $\cC$ be a curve in $\mc{A}$ defined over $\Qbar$ and suppose $\pi(\cC)$ dominates $S$. Then $\cC\cap \mc{A}^{\{2\}}$ is contained in a finite union of flat subgroup schemes of positive codimension.
\end{theorem}

In Section \ref{proofsch} we will see how this theorem is a consequence of our two main Theorems and the previous works \cite{linrel}, \cite{HabPila14}, \cite{Viada2008}, \cite{galateau2010} and \cite{Maurin}.

\section{Preliminaries}\label{periods}

We consider a smooth algebraic curve $S/\C$ and its function field $K=\C(S)$. Let $A$ be an abelian variety defined over $K$ and let $T$ be a torus, $T\cong \Gm^m$. We assume that the largest abelian variety $A_0$, defined over $\C$ and isomorphic over $\overline{K}$ to an abelian subvariety of $A$, is embedded in $A$, and call it the constant part, or $\C$-trace, of $A$. Consider now $G=T\times A$ and set $G_0=T \times A_0$. Here and in the sequel, when necessary we will tacitly restrict $S$ to a non empty open subset which we will still denote by S. Then $G$ defines a family of semiabelian varieties, which we indicate by $G\rightarrow S$. 

We are going to consider our geometrical objects as analytic. When doing so we use the upper index $^{an}$.

Now, our family $G\rightarrow S$ defines an analytic sheaf $G^{an}$ of Lie groups over the Riemann surface $S^{an}$ and its relative Lie algebra $Lie (G)/S$ defines an analytic sheaf $Lie (G^{an})$ over $S^{an}$. Fix a $\Lambda \subseteq S(\C)$ homeomorphic to a closed disk. We have the following exact sequence of analytic sheaves over $\Lambda$ 
$$
0 \longrightarrow \Pi_G \longrightarrow Lie (G^{an}) \xrightarrow{\exp_G}  G^{an} \longrightarrow 0,
$$
see Appendix E of \cite{BerPil}.

We fix a basis for the local system of periods $\Pi_G$ and call $F$ the field generated over $K$ by such basis.
For a local section $\bo{x} \in Lie (G^{an})$ we denote by $\bo{y}=\exp_G(\bo{x})$ its image in $G^{an}$.

\begin{lemma}\label{transc}
Let $\bo{x}\in Lie (G^{an})$ and $\bo{y}=\exp_G(\bo{x})$. Assume moreover that $\bo{y}$ is a $K$-rational point of $G$.
Then, if $\text{tr.deg}_F F(\bo{x})< \dim G$, there exists $H$, a proper algebraic subgroup of $G$, such that $\bo{y} \in H+G_0(\C)$.
\end{lemma}

\begin{proof}
This is a consequence of Th\'eor\`eme L of \cite{Ber09} (see also \cite{Ber11}). The theorem is stated for $G=T\times \tilde{A}$, where $\tilde{A}$ is the universal vectorial extension of $A$. The claim follows by the functoriality of the exponential morphisms, by the fact that $K$-rational points of $A$ and $Lie (A)$ can be lifted to $K$-rational points of $\tilde{A}$ and $Lie (\tilde{A})$ and by a dimension count. Moreover, any algebraic subgroup of $\tilde{A}$ projecting onto $A$ must fill up $\tilde{A}$. Finally, to see that $K$ can be replaced by $F$ in Th\'eor\`eme L, one must look at the formula at the beginning of page 2786.
\end{proof}

We consider $E_\lambda$ as a family over $Y(2)=\mathbb{P}^1 \setminus \{0,1, \infty \}$.
By abuse of notation we indicate by $E_\lambda^n$ the fibered product over $Y(2)$ of $n$ copies of $E_\lambda$.

Our theorems deal with a curve $\cC$ inside a family of semi-abelian varieties $G$ of the following three types:
\begin{enumerate}
\item $G= E^n_\lambda \times E^m_\mu$ with $\lambda$ and $\mu$ both non-constant;
\item $G= E^n_\lambda \times E^m_\mu$ with $\lambda$ non-constant and $\mu =\mu_0 \in \Qbar$;
\item $G= E^n_\lambda \times \Gm^m$ with $\lambda$ non-constant.
\end{enumerate}

For the rest of the paper we will refer to these as cases (1), (2) and (3).

In the first two cases our family has basis $Y(2)\times Y(2)$, but, since we must have a one-dimensional basis in order to apply Lemma \ref{transc}, we will restrict it to $\pi(\cC)$, where $\pi : E^n_\lambda \times E^m_\mu \rightarrow Y(2)\times Y(2) $ is the structural morphism. 

Now, we let $\widehat{\cC}$ be the set of points $\bo{c} \in \cC(\C)$ that do not map to singular points of $\pi(\cC)$, that are not ramified points of $\pi_{|_\cC}$ and such that $\lambda, \mu \neq 0,1$ and $ x_1,\dots ,x_n \neq 0,1, \lambda$ and, in cases (1) and (2), $ u_1, \dots , u_m \neq 0,1,\mu $ on $\bo{c}$. In this way we remove only finitely many algebraic points of $\cC$. We set $S=\pi ( \widehat{\cC})$ and $K=\C(S)$. We can then consider our family of semi-abelian varieties $G$ as a semi-abelian variety defined over the function field $K$.

We now recall a few facts about algebraic subgroups. The following is a well-known fact (see, for instance, Lemma 7 of \cite{MWZero}). 

\begin{lemma}\label{lemsubg1}
Consider the algebraic group $G=E_\lambda^n \times E_\mu^m \times \Gm^l$ and suppose $E_\lambda$ and $E_\mu$ are non-isogenous. Then, any algebraic subgroup of $G$ is of the form $H_1 \times H_2 \times H_3$, where $H_1$ is an algebraic subgroup of $E_\lambda^n$, $H_2$ of $E_\mu^m$ and $H_3$ of $\Gm^l$.
\end{lemma}

Now, let $G=E_\lambda$, $E_{\mu_0}$ (with $\mu_0 \in \C$) or $\Gm$ and $R=End(G)$. We use the additive notation. 

Any $\bo{a} \in R^m$, induces an homomorphism
$$
\begin{array}{lll}
\bo{a}:& G^m& \rightarrow G\\
&(g_1, \dots ,g_m)& \mapsto a_1 g_1+ \cdots +a_m g_m
\end{array}
$$
and we indicate by $ker(\bo{a})$ the kernel of this homomorphism. The following is again a well-known fact (see Fact 5.2 of \cite{JKS} for a proof sketch).
 
\begin{lemma}\label{lemsubg2}
Let $H$ be a proper algebraic subgroup of $G^m$. Then, there exists $\bo{a} \in R^m\setminus \{ 0\}$ such that $H \subseteq ker (\bo{a})$. Moreover, $ker (\bo{a})$ is an algebraic subgroup of $G^m$ of codimension 1.
\end{lemma}

Now, set $G=E_{\mu_0}$ (with $\mu_0 \in \C$) or $\Gm$ and again $R=End(G)$. Let $\bo{a}\in R^m \setminus \{0\}$. Then, any $ker (\bo{a})$ is a finite union of cosets $v+H$, where $v=(v_1,\dots ,v_m)$ has finite order and $H$ is a connected proper subgroup of $G^m$ of codimension 1. For $g=(g_1, \dots ,g_m) \in G^m$ and $a \in R$, we use the notation $ag$ to indicate $(ag_1, \dots , a g_m)$.

\begin{lemma}\label{lemsubg3}
Let $\bo{a}\in R^m$ with $a_h\neq 0 $ for some $h \in \{1, \dots ,m \}$. Then each component of $ker (\bo{a})$ is a coset $v+H$ for some $v \in G^m$ with $a_h v=0$.
\end{lemma}

\begin{proof}
We need to show that each component of $ker (\bo{a})$ contains a $v$ with $a_h v=0$.
Fix a component $g+H$ for $g=(g_1, \dots , g_m)$. The subgroup $H$ is connected and we can consider its Lie algebra $Lie(H)$ as a codimension 1 subspace of $Lie(G^m)$ defined by the equation $a_1 x_1+ \dots + a_m x_m=0$. Fix $z_1, \dots , z_m \in Lie (G)$ with $\exp_G(z_i)=g_i$. Now, since $a_h \neq 0$, there exists $(z_1', \dots , z_m') \in Lie (H)$ such that $z'_i=z_i$ for all $i \neq h$. Then, if $g'=\exp_{G^m}(z_1', \dots , z_m')=(g'_1, \dots, g'_m)$, we have that $g_i=g'_i$ for all $i \neq h$. Therefore, if we set $v=g-g'$, we have $v_i=0$ for all $i\neq h$, but $v \in ker (\bo{a})$. Thus, we have found our element $v \in g+H$ with $a_h v=0$.
\end{proof}

Now, choose $\bo{c}^* \in \widehat{\cC}$ and a neighborhood $N_{\bo{c}^*}$ of $\bo{c}^*$ on $\widehat{\cC}$, mapping injectively to $S$ via $\pi$. Let $D_{\bo{c}^*}$ be a subset of $\pi(N_{\bo{c}^*})$, containing $t^*:=\pi(\bo{c}^*)$ and homeomorphic (via a local analytic isomorphism) to a closed disc.

On $N_{\bo{c}^*}$, and therefore on $D_{\bo{c}^*}$, it is possible to define analytic $f,g,z_1,\dots ,z_n$ such that $\{f,g \}$ is a basis for the local system of periods $\Pi_{E_\lambda}$ and, for all $t \in D_{\bo{c}^*} $, we have $\exp_{E_{\lambda(\bo{c})}}(z_i(t))= P_i(\bo{c})$, where $\bo{c}$ is the unique point of $N_{\bo{c}^*} \cap \pi^{-1}(t)$. For this see Section 5 of \cite{linrel} or Section 3 of \cite{MZ14a}.

Analogously, we can define analytic $h,k,w_1,\dots ,w_m$ such that $\{h,k \}$ is a basis for the local system of periods $\Pi_{E_\mu}$ and we have $\exp_{E_{\mu(\bo{c})}}(w_j(t))= Q_j(\bo{c})$.

In case (3), $\Pi_{\Gm}$ has rank 1 and we choose $\{2 \pi i \}$ as a basis. We define $w_1,\dots ,w_m$ to be principal determination of the complex logarithm, i.e. $w_j(t)=\log \rho_j + 2 \pi i \theta_j$ where $u_j=\rho_j e^{2\pi i \theta_j}$ and $\theta_j \in [0,1)$.

\begin{corollary}\label{cortr}
In case (1), under the hypotheses of Theorem \ref{mainthm}, we have that $z_1, \dots ,z_n,$ $ w_1, \dots ,w_m$ are algebraically independent over $\C(f,g,h,k)$.
\end{corollary}

\begin{proof}
In case (1) we have $A_0=0$ and there is no toric part. Therefore, if $z_1, \dots ,z_n, w_1, \dots ,w_m$ were algebraically dependent, then $(P_1, \dots ,P_n, Q_1, \dots , Q_m)$ would lie in an algebraic subgroup of $E^n_\lambda \times E^m_\mu$. Therefore, by Lemma \ref{lemsubg1} and \ref{lemsubg2}, there would be an identical relation among the $P_i$ or the $Q_j$ contradicting the hypotheses of Theorem \ref{mainthm}.
\end{proof}

\section{O-minimality and point counting}

For the basic properties of o-minimal structures we refer to \cite{vandenDries1998} and \cite{DriesMiller}.

\begin{definition}
A \textit{structure} is a sequence $\mathcal{S}=\left( \mathcal{S}_N\right)$, $N\geq 1$, where each $\mathcal{S}_N$ is a collection of subsets of $\R^N$ such that, for each $N,M \geq 1$:
\begin{enumerate}
\item $\mathcal{S}_N$ is a boolean algebra (under the usual set-theoretic operations);
\item $\mathcal{S}_N$ contains every semialgebraic subset of $\R^N$;
\item if $A\in \mathcal{S}_N$ and $B\in \mathcal{S}_M$, then $A\times B \in \mathcal{S}_{N+M}$;
\item if $A \in \mathcal{S}_{N+M}$, then $\pi (A) \in \mathcal{S}_N$, where $\pi :\R^{N+M}\rightarrow \R^N$ is the projection onto the first $N$ coordinates.
\end{enumerate}
If $\mathcal{S}$ is a structure and, in addition,
\begin{enumerate}
\item[(5)] $\mathcal{S}_1$ consists of all finite union of open intervals and points,
\end{enumerate}
then $\mathcal{S}$ is called an \textit{o-minimal structure}.
\end{definition}
Given a structure $\mathcal{S}$, we say that $S \subseteq \R^N$ is a \textit{definable set} if $S\in \mathcal{S}_N$. 

Let $U\subseteq \R^{M+N}$. For $t_0\in \R^M$, we set $U_{t_0}=\{ x\in \R^N: (t_0,x) \in U \}$ and call $U$ a \emph{family} of subsets of $\R^N$, while $U_{t_0}$ is called the \emph{fiber} of $U$ above $t_0$. If $U$ is a definable set, then we call it a \emph{definable family} and one can see that the fibers $U_{t_0}$ are definable sets too. Let $S\subseteq \R^N$ and $f:S\rightarrow \R^M$ be a function. We call $f$ a \emph{definable function} if its graph $\lg (x,y) \in S\times \R^{M}:y=f(x) \rg$ is a definable set. It is not hard to see that images and preimages of definable sets via definable functions are still definable.

There are many examples of o-minimal structures, see \cite{DriesMiller}. In this article we are interested in the structure of \textit{globally subanalytic sets}, usually denoted by $\R_{\text{an}}$. We are not going to pause on details about this structure because it is enough for us to know that, if $D\subseteq \R^N$ is a compact definable set, $I$ is an open neighborhood of $D$ and $f:I\rightarrow \R^M$ is an analytic function, then $f(D)$ is definable in $\R_{\text{an}}$.

We now fix an o-minimal structure $\mathcal{S}$. 

\begin{proposition}[\cite{DriesMiller}, 4.4] \label{unifbound}
Let $U$ be a definable family. There exists a positive integer $\gamma$ such that each fiber of $U$ has at most $\gamma$ connected components.
\end{proposition}

We are going to use a result from \cite{HabPila14}. For this we need to define the height of a rational point. The height used in \cite{HabPila14} is not the usual projective Weil height, but a coordinatewise affine height. If $a/b$ is a rational number written in lowest terms, then $H(a/b)= \max \{|a|,|b|\}$ and, for an $N$-tuple $(\alpha_1, \dots , \alpha_N) \in \Q^N$, we set $H(\alpha_1, \dots , \alpha_N)= \max_i H(\alpha_i)$.
For a family $Z \subseteq \R^{M_1+M_2+N}$, a positive real number $T$ and $t \in \R^{M_1}$ we define
\begin{equation}\label{def}
Z_t^{\sim}(\Q,T)=\lg (y,z)  \in  Z_t: y\in \Q^{M_2}, H(y) \leq T \rg.
\end{equation}

By $\pi_1$ and $\pi_2$ we indicate the projections of $Z_t$ to the first $M_2$ and the last $N$ coordinates respectively.

\begin{proposition}[\cite{HabPila14}, Corollary 7.2]\label{HabPila}
Let $Z\subseteq \R^{M_1+M_2+N}$ be a definable family.
For every $\epsilon>0$ there exists a constant $c=c(Z,\epsilon)$ with the following property. Fix $t \in \R^{M_1}$ and $T\geq 1 $. If $|\pi_2(\Sigma)|> c T^\epsilon$ for some $\Sigma \subseteq Z_t^{\sim}(\Q,T)$, then there exists a continuous definable function $\delta:[0,1] \rightarrow Z_t$ such that
\begin{enumerate}
\item the composition $\pi_1 \circ \delta : [0,1] \rightarrow \R^{M_2}$ is semi-algebraic and its restriction to $(0,1)$ is real analytic;
\item the composition $\pi_2 \circ \delta : [0,1] \rightarrow \R^N$ is non-constant;
\item we have $\pi_2(\delta(0)) \in \pi_2 (\Sigma) $.
\end{enumerate}
\end{proposition}

\section{The main estimate}

Fix a $\bo{c} \in \widehat{\cC}$ and a neighborhood $N_{\bo{c}}$ of $\bo{c}$ on $\widehat{\cC}$. Moreover, fix a closed disc $D_{\bo{c}}$ inside $\pi(N_{\bo{c}})$, centered in $\pi(\bo{c})$ and analytically isomorphic to a closed disc. In Section \ref{periods} we defined the analytic functions $f,g,h,k,z_1, \dots ,z_n,w_1, \dots ,w_m$ on $D_{\bo{c}}$ as a basis for the local system of periods of $E_\lambda$ and $E_\mu$ (or $\Gm$) and elliptic logarithms of the $P_i$ and $Q_j$ (or logarithms of the $u_j$).

For the rest of this section we suppress the dependence on $\bo{c}$ in the notation, since it is fixed. We use Vinogradov's $\ll$ notation. The implied constant is always going to depend on $D$.

In cases (1) and (2), we define, for $\bo{a} \in \Z^n\setminus \{0\}$ and $\bo{b} \in R_2^m\setminus \{0\}$,
\begin{equation}\label{defD}
D(\bo{a},\bo{b})=\lg t \in D: \sum a_i z_i(t) \in \Z f(t)+\Z g(t) \text{ and } \sum b_j w_j(t) \in \Z h(t)+\Z k(t) \rg .
\end{equation}

In case (3), for $\bo{a} \in \Z^n\setminus \{0\}$ and $\bo{b} \in \Z^m\setminus \{0\}$, we set
\begin{equation*}\label{defDm}
D(\bo{a},\bo{b})=\lg t \in D: \sum a_i z_i(t) \in \Z f(t)+\Z g(t) \text{ and } \sum b_j w_j(t) \in  2\pi i \Z\rg .
\end{equation*}

For a vector of integers $\bo{a}$, we indicate by $|\bo{a}|$ its max norm $\max \{ |a_1|, \dots , |a_n| \}$. In case (2), if $E_{\mu_0}$ has CM, we have $R_2= \Z+ \rho \Z$, for some quadratic integer $\rho$. For $\bo{b}=(b_1,\dots , b_m) \in R_2^m$,
we set $|\bo{b}|=\max \{ |N(b_1)|, \dots , |N(b_m)|  \}$, where $N(b_j )$ is the norm of $b_j$.

\begin{proposition}\label{estim}
Under the hypotheses of Theorem \ref{mainthm} and Theorem \ref{mainthm2}, for every $\epsilon>0$ we have $|D(\bo{a},\bo{b})|\ll_\epsilon (\max \{ |\bo{a}|,|\bo{b}|\})^\epsilon$, for every non-zero $\bo{a}, \bo{b}$.
\end{proposition}

We are going to prove this proposition in cases (1), (2) and (3) separately. Let us first collect a few definitions and facts needed for all three of them.

Define
$$
\Delta=f \overline{ g }-\overline{f} g,
$$
which does not vanish on $D$, since $f(t)$ and $g(t)$ are $\R$-linearly independent for every $t \in D$. Moreover, let
\begin{align*}
p_i=\frac{z_i \overline{g}-\overline{z_i} g}{\Delta},   \   \   \   q_i=-\frac{z_i \overline{ f}-\overline{z_i} f}{\Delta}.
\end{align*}
One can easily check that these are real-valued and, furthermore, we have
$$
z_i= p_i f+ q_i g.
$$
If we view $D$ as a subset of $\R^2$, then $p_i$ and $q_i$ are real analytic functions on a neighborhood of $D$.

Analogously, in cases (1) and (2), we can define the real valued functions $r_j$, $s_j$ with
$$
w_j= r_j h+ s_j k.
$$
In case (3) we set
$$
w_j=r_j+2 \pi i s_j, 
$$
where again $r_j$ and $s_j$ are real valued.

In all cases we define
\begin{align*}
\Theta : D& \rightarrow \R^{2n+2m}\\
 t &\mapsto (p_{1}(t),q_{1}(t),\dots ,p_{n}(t),q_{n}(t), r_1(t), s_1(t), \dots, r_m(t), s_m(t) ),
\end{align*}
and set $S=\Theta(D)$.

Since $\Theta$ is analytic and $D$ is a closed disc we have that $S$ is definable in $\R_{\text{an}}$.

\begin{lemma}\label{lemfin}
Under the hypotheses of Theorem \ref{mainthm} and Theorem \ref{mainthm2}, there exists a constant $\gamma$ (depending only on $D$) such that, for every choice of integers $a_1, \dots, a_{n+2}$, not all zero, the number of $t$ in $D$ with
\begin{equation}\label{1rel}
a_{1}z_{1}(t)+\dots + a_{n}z_{n}(t)=a_{n+1}f(t)+ a_{n+2} g(t).
\end{equation}
is at most $\gamma$.
\end{lemma}

\begin{proof}

First, suppose there is an infinite set $E\subseteq D$ on which, for every $t \in E$, \eqref{1rel} holds for some fixed $a_1, \dots, a_{n+2}$, not all zero. Since this is a set with an accumulation point, the same relation must hold on the whole $D$ (see Ch.~ III, Theorem 1.2 (ii) of \cite{LangCompl}), contradicting the hypotheses of Theorems \ref{mainthm} and \ref{mainthm2}.

The existence of a uniform bound $\gamma$ follows from Proposition \ref{unifbound} and the fact that $\Theta$ is a definable function.
\end{proof}

In what follows, $(p_{1},q_{1},\dots , r_m, s_m )$ will indicate coordinates in $ \R^{2n+2m}$. 

We now consider the three cases separately.

\subsection{Case (1)}

We start considering case (1), i.e., our curve lies in $E_\lambda^n \times E_\mu^m$ and $\lambda$ and $\mu$ are both not constant.

For $T>0$, we call $S^{(1)}(\bo{a},\bo{b},T)$ the set of points of $S$ of coordinates $(p_{1},q_{1},\dots , r_m, s_m )$ such that there exist $a_{n+1},a_{n+2},b_{m+1},b_{m+2} \in \Z \cap [-T,T ]$ with
\begin{equation}\label{system}
\lg
\begin{array}{c}
a_{1}p_{1}+\dots + a_{n} p_{n}=a_{n+1},\\
a_{1}q_{1}+\dots + a_{n} q_{n}=a_{n+2},\\
b_{1}r_{1}+\dots + b_{m} r_{m}=b_{m+1},\\
b_{1}s_{1}+\dots + b_{m}s_{m}=b_{m+2}.
\end{array}
\right.
\end{equation} 

\begin{lemma}\label{lemest}
Under the hypotheses of Theorem \ref{mainthm}, for every $\epsilon>0$ we have $$|S^{(1)}(\bo{a},\bo{b},T)|\ll_\epsilon (\max \{|\bo{a}|,| \bo{b}|,  T\})^\epsilon,$$
for all non-zero $\bo{a}$ and $\bo{b}$ and all $T\geq 1$.
\end{lemma}

\begin{proof}
Set $T'=\max \{ |\bo{a}|,|\bo{b}|,T \}$ and fix $\epsilon>0$.

Define $W$ to be the set of $\left(\alpha_1, \dots ,\alpha_{n+2},\beta_1, \dots , \beta_{m+2}, p_1, \dots , s_m  \right)  \in  \R^{n+2+m+2}\times S$ such that
\begin{equation}\label{sys2}
\lg
\begin{array}{c}
\alpha_{1}p_{1}+\dots + \alpha_{n} p_{n}=\alpha_{n+1},\\
\alpha_{1}q_{1}+\dots + \alpha_{n} q_{n}=\alpha_{n+2},\\
\beta_{1}r_{1}+\dots + \beta_{m} r_{m}=\beta_{m+1},\\
\beta_{1}s_{1}+\dots + \beta_{m}s_{m}=\beta_{m+2}.
\end{array}
\right.
\end{equation}
This is a definable set in $\R_{an}$. Recall the notation introduced in \eqref{def}. The set $W^{\sim}(\Q,T')$ consists of those tuples $\left(\alpha_1, \dots ,\alpha_{n+2} , \beta_1,  \dots , \beta_{m+2}, p_1, \dots , s_m  \right)  \in  \R^{n+2+m+2}\times S$ with rational $\alpha_1, \dots ,\alpha_{n+2} , \beta_1, \dots , \beta_{m+2}$ of height at most $T'$.
We set $\Sigma=W^{\sim}(\Q,T')$ and note that $\pi_2(\Sigma)\supseteq S^{(1)}(\bo{a},\bo{b},T)$, where $\pi_2:W\rightarrow S$ is the projection to $S$.
Then, $|S^{(1)}(\bo{a},\bo{b},T)| \leq |\pi_2(\Sigma)|$. We claim that $|\pi_2 (\Sigma)|\ll_\epsilon (T')^\epsilon$. Suppose not. Then, by Proposition \ref{HabPila}, there exists a continuous definable $\delta:[0,1] \rightarrow W$ such that $\delta_1:=\pi_1 \circ \delta : [0,1] \rightarrow \R^{n+2+m+2}$ is semi-algebraic and $\delta_2:=\pi_2 \circ \delta : [0,1] \rightarrow S$ is non-constant. Therefore, there is a connected infinite subset $E\subseteq [0,1]$ such that $\delta_1(E)$ is contained in a real algebraic curve and $\delta_2(E)$ has positive dimension. Then, there exists a connected infinite $D'\subseteq D$ with $ \Theta(D') \subseteq \delta_2(E)$.

The coordinate functions $\alpha_{1}, \dots , \alpha_{n+2},\beta_1, \dots ,\beta_{m+2}$ on $D'$ satisfy $n+m+3$ independent algebraic relations with coefficients in $\C$. Moreover, we have the relations given by \eqref{sys2},
which translate to 
\begin{equation*}
\lg
\begin{array}{c}
\alpha_{1}z_{1}+\dots + \alpha_{n} z_{n}=\alpha_{n+1}f +\alpha_{n+2} g, \\
\beta_{1}w_1+\dots + \beta_{m} w_{m}=\beta_{m+1} h+\beta_{m+2} k,
\end{array}
\right.
\end{equation*}
adding 2 algebraic relations among the $\alpha_{1}, \dots , \alpha_{n+2},\beta_1, \dots ,\beta_{m+2}$, the $z_i$, the $w_j$, $f$, $g$, $h$ and $k$.

Thus, on $D'$, and therefore by continuation on the whole $D$, the $n+2+m+2+n+m$ functions $\alpha_{1}, \dots , \alpha_{n+2},\beta_1, \dots ,\beta_{m+2},z_{1}, \dots , z_{n},w_1, \dots ,w_{m}$ satisfy $n+m+3+2$ independent algebraic relations over $F=\C(f,g,h,k)$. Thus,
$$
\textsl{tr.deg}_F F ({z}_1, \dots ,{z}_n,w_1, \dots ,w_m )\leq n+m-1.
$$
This contradicts Corollary \ref{cortr}, and proves the claim and the lemma.
\end{proof}

If $t\in D(\bo{a},\bo{b})$, then $\Theta(t)$ satisfies \eqref{system} for some integers $a_{n+1},a_{n+2},b_{m+1},b_{m+2}$. Now, since $D$ is compact, we have that the sets $z_i(D), w_j(D), f(D),$ $ g(D) ,h(D), k(D)$ are bounded and therefore we can choose  $a_{n+1},a_{n+2},b_{m+1},b_{m+2}$ bounded solely in terms of $|\bo{a}|$ and $|\bo{b}|$. Therefore, we have $\Theta (t) \in S^{(1)}(\bo{a},\bo{b},T_0)$, with $T_0\ll \max \{ |\bo{a}|,|\bo{b}|\}$. Now, by Lemma \ref{lemfin} we have $|D(\bo{a},\bo{b})|\ll |S^{(1)}(\bo{a},\bo{b}, T_0)|$ and the claim of Proposition \ref{estim} follows from Lemma \ref{lemest}. 

\subsection{Case (2)} Case (2) deals with a curve $\cC$ inside $E_\lambda^n \times E_{\mu_0}^m$ with $\lambda$ not constant and $\mu_0 \in \Qbar$.

For all $\bo{b} \in R_2^m\setminus \{ 0\} $ there is a codimension 1 abelian subvariety $Z$ of $E_{\mu_0}^m$, depending only on $\bo{b}$, such that, if a point $(Q_1, \dots, Q_m)\in E_{\mu_0}^m$ satisfies the relation $b_1 Q_1+ \dots + b_m Q_m=O$, then it is contained in some coset $R+Z$, where $R$ a torsion point of $E_{\mu_0}^m$. We let $X=E_{\mu_0}^m/Z$. This is a 1-dimensional abelian variety and we set $\phi:E_{\mu_0}^m \rightarrow X$ to be the quotient morphism. This induces the linear map $d \phi:Lie (E_{\mu_0}^m) \rightarrow Lie (X)$. If we identify $Lie (E_{\mu_0}^m)$ with $\C^m$ and $ Lie (X)$ with $\C$, then $d \phi$ corresponds to a complex vector $\bo{l} \in \C^m$ acting on $\C^m$ as a scalar product. Note that $\bo{l}$ depends only on $Z$ and therefore on $\bo{b}$.

We set $Q_j(t)=\exp_{E_{\mu_0}}(w_j(t))$. For $t\in D$, if $(Q_1(t), \dots , Q_m(t))\in R+Z$, with $R$ of finite order, then $\phi (Q_1(t), \dots , Q_m(t))=\phi (R)$ and there are $d_1,e_1,\dots, d_m,$ $e_m \in \Q$ with $\exp_{E_{\mu_0}^m}(d_1 h + e_1 k, \dots ,d_m h + e_m k)=R$ and $$d \phi (w_1(t), \dots ,w_m(t)) = d \phi  (d_1 h + e_1 k, \dots ,d_m h + e_m k).$$ 

We define $S^{(2)}(\bo{a},\bo{b},T)$ to be the set of points of $S$ of coordinates $(p_{1},q_{1},\dots , r_m, s_m )$ such that there exist $a_{n+1},a_{n+2} \in \Z \cap [-T,T ]$ and $d_1,e_1,\dots, d_m,e_m \in \Q$ of height at most $T$ with
\begin{equation}\label{system2}
\lg
\begin{array}{c}
a_{1}p_{1}+\dots + a_{n} p_{n}=a_{n+1},\\
a_{1}q_{1}+\dots + a_{n} q_{n}=a_{n+2},\\
\bo{l} \cdot (r_1 h + s_1 k, \dots , r_m h + s_m k) = \bo{l} \cdot (d_1 h + e_1 k, \dots ,d_m h + e_m k).
\end{array}
\right.
\end{equation} 

In the following lemma we are going to see $\bo{l}$ as a vector in $\R^{2m}$. The last equation above is an equality of complex numbers but it corresponds to two equalities of real numbers (recall that $h$ and $k$ are fixed complex numbers in this case).

\begin{lemma}\label{lemest2}
Under the hypotheses Theorem \ref{mainthm}, for every $\epsilon>0$ we have $$|S^{(2)}(\bo{a},\bo{b},T)|\ll_\epsilon (\max \{|\bo{a}|,  T\})^\epsilon$$
for all non-zero $\bo{a}$ and $\bo{b}$ and all $T>0$.
\end{lemma}

\begin{proof}
Set $T'=\max \{ |\bo{a}|,T \}$ and fix $\epsilon>0$.

Define $W$ to be the set of $\left(\bo{\nu},\alpha_1,\dots ,\alpha_{n+2},\chi_1,\psi_1,\dots, \chi_m,\psi_m, p_1, \dots , s_m  \right)  \in  \R^{2m +n+2+2m}\times S$, with
\begin{equation*}\label{sys3}
\lg
\begin{array}{c}
\alpha_{1}p_{1}+\dots + \alpha_{n} p_{n}=\alpha_{n+1},\\
\alpha_{1}q_{1}+\dots + \alpha_{n} q_{n}=\alpha_{n+2},\\
\bo{\nu} \cdot (r_1 h + s_1 k, \dots , r_m h + s_m k) =\bo{\nu} \cdot (\chi_1 h + \psi_1 k, \dots ,\chi_m h + \psi_m k).
\end{array}
\right.
\end{equation*}

This is a definable set in $\R_{an}$. We consider the fiber $W_{\bo{l}}$, where $\bo{l}$ is associated to $\bo{b}$ as explained earlier.

We set $\Sigma= (\Z^{n+2} \times \Q^{2m} \times S) \cap     W_{\bo{l}}^{\sim}(\Q,T')$, and note that  $\pi_2(\Sigma ) \supseteq S^{(2)}(\bo{a},\bo{b},T) $ where $\pi_2$ is the projection on $S$.
Then, $|S^{(2)}(\bo{a},\bo{b},T)| \leq |\pi_2(\Sigma)|$. We claim that $|\pi_2(\Sigma)|\ll_\epsilon (T')^\epsilon$, where the implied constant is independent of $\bo{l}$ and therefore independent of $\bo{b}$.  Suppose not, then by Proposition \ref{HabPila} there exists a continuous definable $\delta:[0,1] \rightarrow W_{\bo{l}}$ such that $\delta_1:=\pi_1 \circ \delta : [0,1] \rightarrow \R^{n+2+2m}$ is semi-algebraic and the composition $\delta_2:=\pi_2 \circ \delta : [0,1] \rightarrow S$ is non-constant. Moreover, $\delta_2 (0) \in   \pi_2(\Sigma)$. Therefore, there is a connected infinite subset $E\subseteq [0,1]$, such that $\delta_1(E)$ is contained in a real algebraic curve and $\delta_2(E)$ has positive dimension. Thus, there exists a connected infinite $D'\subseteq D$ with $ \Theta(D') = \delta_2(E)$. Moreover, there is $t_0 \in D$ with $\Theta(t_0)=\delta_2(0)$. Then, since $\delta_2(0) \in \pi_2(\Sigma)$, the point $(Q_1(t_0), \dots ,Q_m(t_0)) \in R+Z$ for some torsion point $R \in E_{\mu_0}^m$.

Now, on $D'$ we have that $\alpha_1,\dots , \alpha_{n+2},\chi_1, \dots ,\psi_m$ are $n+2+2m$ functions that generate an extension of transcendence degree at most 1 over $\C$. Moreover, $d \phi (w_1, \dots ,w_m)= d \phi (\chi_1 h + \psi_1 k, \dots , \chi_m h + \psi_m k)$ and note that $d \phi$ is a linear map.

Therefore, $\alpha_1, \dots ,\alpha_{n+2}$ and $d \phi (w_1, \dots ,w_m)=w'$ are $n+3$ functions on $D'$ satisfying $n+2$ algebraic relations over $\C$. Moreover, we have $\alpha_1 z_1+ \dots +\alpha_nz_n = \alpha_{n+1} f + \alpha_{n+2} g$.
Then, the $2n+3$ functions $z_1, \dots, z_n, \alpha_1, \dots, \alpha_{n+2},w'$ satisfy $n+3$ independent relations over $F=\C(f,g)$ on $D'$ and these extend on $D$. Therefore,
$$
\textsl{tr.deg}_F F ({z}_1, \dots ,{z}_n,w')\leq n,
$$
on $D$.

Now, we want to apply Lemma \ref{transc} to $G=E_\lambda^n \times X$ which has dimension $n+1$ and $G_0=\{(O,\dots, O) \} \times X$. Then, on $D$, the lemma implies that $$\exp_G(z_1,\dots, z_n,w')=(P_1,\dots , P_n, \phi(Q_1, \dots, Q_m))\in H+G_0(\C),$$ for some proper algebraic subgroup $H$ of $G$. Since the $P_i$ are independent and $X$ has dimension 1, we have that $H=E_\lambda^n \times X'$, where $X'$ is a torsion subgroup of $X$. Then
$\phi(Q_1(D), \dots ,Q_m(D))= \{Q' \}$ for some $Q' \in X(\C)$. But recall that there is $t_0 \in D$ with $\phi(Q_1(t_0), \dots, Q_m(t_0))=\phi (R)$ for some torsion point $R$ of $E_{\lambda_0}^m$. Then, $Q'=\phi(R)$ and therefore we have $(Q_1(D), \dots, Q_m(D)) \subseteq R+Z$. This contradicts the hypotheses of Theorem \ref{mainthm}, proving the claim and the lemma.
\end{proof}

\begin{lemma}\label{lemest3}
There exists $T_0\ll \max \{ |\bo{a}|,|\bo{b}|\}$ such that, if $t\in D(\bo{a},\bo{b})$,  then $\Theta(t) \in S^{(2)}(\bo{a},\bo{b},T_0)$.
\end{lemma}

\begin{proof}

Fix $t\in D(\bo{a},\bo{b})$; then, $\Theta(t)$ satisfies \eqref{system2} for some integers $a_{n+1},a_{n+2}$ and rationals $d_1, e_1, \dots ,$ $d_m, e_m$. Now, since $D$ is compact, as before we have that the sets $z_i(D), f(D),$ $ g(D)$ are bounded and therefore we can choose  $a_{n+1},a_{n+2}$ bounded solely in terms of $|\bo{a}|$.

We need to prove that we can choose rationals $d_1, e_1, \dots ,d_m, e_m$ of height $\ll |\bo{b}|$ with
$$
(w_1(t),\dots ,w_m(t))-(d_1 h + e_1 k, \dots ,d_m h + e_m k) \in Lie( Z).
$$
We have that $(Q_1(t),\dots, Q_m(t))=\exp_{E_{\mu_0}^m}(w_1(t),\dots ,w_m(t))\in R+Z$, where $Z$ is the unique abelian subvariety of $E_{\mu_0}^m$ associated to the vector $\bo{b}$ as explained above and $R$ is a torsion point of $E_{\mu_0}^m$. Since $(Q_1(t),\dots, Q_m(t)) \in \ker (\bo{b})$, by Lemma \ref{lemsubg3}, we can suppose that $R$ has order at most $|\bo{b}|$ .

Let $\bo{w}=(w_1(t),\dots ,w_m(t))$. We know that there are rationals $d'_1, e'_1, \dots ,d'_m, e'_m \in [0,1)$ with  $\exp_{E_{\mu_0}^m} (d'_1 h + e'_1 k, \dots ,d'_m h + e'_m k)=R$. Therefore, $d'_1, e'_1, \dots ,d'_m, e'_m$ have denominators $\ll |\bo{b}|$ and we have
$$
\bo{w} - (d'_1 h + e'_1 k, \dots ,d'_m h + e'_m k) \in \Pi_{E_{\mu_0}^m}+ Lie (Z).
$$
 
We call $\bo{c}'=(d'_1 h + e'_1 k, \dots ,d'_m h + e'_m k)$.
We indicate by $\Vert \cdot \Vert$ the max norm on $Lie (E_{\mu_0}^m)=\C^m$. Note that $\Vert \bo{w}-\bo{c}' \Vert \ll 1$. Let $\bo{\eta}\in  \Pi_{E_{\mu_0}^m}$ and $\bo{x} \in Lie (Z)$ be such that $\bo{w}-\bo{c}'=\bo{\eta}+\bo{x}$.  The subspace $Lie (Z)$ is defined by the equation $b_1 w_1+ \dots +b_mw_m=0$. We can suppose $b_1\neq 0$. Consider the following $2(m-1)$ vectors: $\bo{\eta}_1=(b_2 h, -b_1 h,0 , \dots ,0)$, $\bo{\eta}_2=(b_2 k, -b_1 k,0 , \dots ,0)$, $\bo{\eta}_3=(b_3 h,0, -b_1 h,0 , \dots ,0)$, \dots , $\bo{\eta}_{2(m-1)}=(b_m k, 0 , \dots ,0, -b_1 k)$. These are $\R$-linearly independent elements of $\Pi_{E_{\mu_0}^m}$ whose $\R$-span is $Lie (Z)$. Then, there are $\alpha_1, \dots , \alpha_{2(m-1)} \in [0,1)$ with $\bo{x}=\bo{\eta}'+\bo{x}'$, where $\bo{x}'=\sum_{i=1}^{2(m-1)}\alpha_i \bo{\eta}_i \in Lie (Z)$ and $\bo{\eta}' \in \Pi_{E_{\mu_0}^m}$. Note that $\Vert \bo{x}' \Vert \ll |\bo{b}|$.

Finally, we have $\bo{w}-\bo{c}'=\bo{\eta}+\bo{\eta}'+\bo{x}'$ and 
$$
\Vert \bo{\eta}+\bo{\eta}' \Vert \leq \Vert\bo{w}-\bo{c}'\Vert+\Vert\bo{x}'\Vert\ll |\bo{b}|.
$$
If we set $\bo{\eta}+\bo{\eta}'+\bo{c}'=(d_1 h + e_1 k, \dots ,d_m h + e_m k)$, we have just found our rationals of height $\ll |\bo{b}|$ such that
$$
\bo{w}-(d_1 h + e_1 k, \dots ,d_m h + e_m k) \in Lie (Z).
$$
\end{proof}
 
By Lemma \ref{lemfin} we have $|D(\bo{a},\bo{b})|\ll |S^{(2)}(\bo{a},\bo{b}, T_0)|$ and the claim of Proposition \ref{estim} follows from Lemma \ref{lemest2}. 

\subsection{Case (3)}

To deal with case (3), (curve in $E_\lambda^n \times \Gm^m$, $\lambda$ not constant), one can follow the same line as case (2). Here, one has that, for all $\bo{b} \in \Z^m\setminus \{ 0\} $, there is a codimension 1 subtorus $Z$ of $\Gm^m$, depending only on $\bo{b}$, such that, if a point $(u_1, \dots, u_m)\in \Gm^m$ satisfies $u_1^{b_1} \cdots u_m^{b_m}=1$, then it is contained in some coset $RZ$, where $R$ a torsion point of $\Gm^m$ of order at most $|\bo{b}|$. Let $X=\Gm^m/Z$. This is a 1-dimensional torus and we set again $\phi:\Gm^m \rightarrow X$ to be the quotient morphism. This induces the linear map $d \phi:Lie (\Gm^m) \rightarrow Lie (X)$ which again corresponds to a complex vector $\bo{l} \in \C^m$ acting on $\C^m$ as a scalar product.

We define $S^{(3)}(\bo{a},\bo{b},T)$ to be the set of points of $S$ of coordinates $(p_{1},q_{1},\dots , r_m, s_m )$ such that there exist $a_{n+1},a_{n+2} \in \Z \cap [-T,T ]$ and $d_1,\dots, d_m \in \Q$ of height at most $T$ with
\begin{equation*}\label{system5}
\lg
\begin{array}{c}
a_{1}p_{1}+\dots + a_{n} p_{n}=a_{n+1},\\
a_{1}q_{1}+\dots + a_{n} q_{n}=a_{n+2},\\
\bo{l} \cdot (r_1  +2 \pi i s_1 , \dots , r_m + 2 \pi i s_m) = \bo{l} \cdot (2 \pi i d_1 , \dots , 2 \pi i d_m ).
\end{array}
\right.
\end{equation*} 

Following the same line it is possible to prove the analogous of Lemma \ref{lemest2} and \ref{lemest3} and to obtain the claim of Proposition \ref{estim} in case (3) using again Lemma \ref{lemfin}.

\section{Small generators of the relations lattices}

In this section we prove general facts about linear relations on elliptic curves and multiplicative relations on $\Gm$.

For a point $(\alpha_1, \dots ,\alpha_N) \in \overline{\Q}^N$, the absolute logarithmic Weil height $h(\alpha_1, \dots ,\alpha_N)$ is defined by
$$
h(\alpha_1, \dots ,\alpha_N)= \frac{1}{[\Q(\alpha_1, \dots ,\alpha_N):\Q]} \sum_v \log \max \{1,| \alpha_1|_v, \dots ,|\alpha_N|_v\},
$$
where $v$ runs over a suitably normalized set of valuations of $\Q(\alpha_1, \dots ,\alpha_N)$.

Let $\theta$ be an algebraic number and consider the Legendre curve $E=E_\theta$ defined by the equation
$ Y^2=X(X-1)(X-\theta)$. Moreover, let $P_1, \dots , P_n$ be points on $E$, linearly dependent over $\Z$, defined over some finite extension $K$ of $\Q(\theta)$ of degree $\kappa=[K:\Q]$. Suppose that $P_1, \dots , P_n$ have N\'eron-Tate height $\widehat{h}$ at most $q\geq 1$
(for the definition of N\'eron-Tate height, see for example p.\,255 of \cite{Masser88}).
We define
$$
L(P_1, \dots , P_n)=\{ (a_1, \dots , a_n) \in \Z^n: a_1P_1+\dots +a_nP_n= O \}.
$$
This is a sublattice of $\Z^n$ of some positive rank $r$. We want to show that $L(P_1, \dots , P_n)$ has a set of generators with small max norm 
$|\bo{a}|=\max \{|a_1|, \dots ,|a_n|\}$.

\begin{lemma}[\cite{linrel}, Lemma 6.1]\label{lemgeneratorsell}
Under the above hypotheses, there are generators $\bo{a}_1, \dots , \bo{a}_r$ of $L(P_1, \dots , P_n)$ with
$$
|\bo{a}_i|\leq \gamma_1  \kappa^{\gamma_{2}} (h(\theta)+1)^{2n} q^{\frac{1}{2}(n-1)},
$$
for some positive constants $\gamma_1, \gamma_2$ depending only on $n$.
\end{lemma}

Analogously, consider a vector $(\alpha_1, \dots , \alpha_m) \in (K\setminus \{0\})^m$, for some number field $K$, with $\kappa=[K:\Q]$, as above. Suppose the $\alpha_j$ are multiplicatively dependent. We define
$$
L(\alpha_1, \dots , \alpha_m)=\{ (b_1, \dots , b_m) \in \Z^m: \alpha_1^{b_1} \dots  \alpha_m^{b_m}= 1 \}.
$$
Fix $h \geq 1$ with $h(\alpha_j)\leq h$ for all $j=1,\dots ,m$.

\begin{lemma}\label{lemgeneratorsmult}
Under the above hypotheses, there are generators $\bo{b}_1, \dots , \bo{b}_r$ of $L(\alpha_1, \dots , \alpha_m)$ with
$$
|\bo{b}_i|\leq \gamma_3  \kappa^{\gamma_{4}}  h^{m-1},
$$
for some positive constants $\gamma_3, \gamma_4$ depending only on $m$.
\end{lemma}

\begin{proof}
Suppose first that not all the $\alpha_j$ are roots of unity.
By Theorem $\Gm$ of \cite{Masser88}, if $\alpha_1, \dots, \alpha_m$ are multiplicatively dependent algebraic numbers of height at most $h \geq \eta$, 
then $L(\alpha_1, \dots , \alpha_m)$ is generated by vectors with max norm at most 
$$
m^{m-1} \omega \left(\frac{h}{\eta}\right)^{m-1},
$$
where $\omega$ is the number of roots of unity in $K$ and $\eta=\inf h(\alpha)$, for $\alpha\in K\setminus \{0\}$ not a root of unity.
We need to bound $\omega$ and $\eta$. The constants $\gamma_5, \dots, \gamma_8$ are absolute constants.

The first bound is elementary since the roots of unity in $K$ form a cyclic group generated by, say, $\zeta_N$ a primitive $N$-th root of unity. We must then have $\phi(N)\leq \kappa$ ($\phi$ indicates the Euler function) and we know $\phi(N) \geq \gamma_5 \sqrt{N}$. Therefore we can take 
\begin{equation}\label{tor}
\omega \leq \gamma_6 \kappa^2.
\end{equation}

For $\eta$, an estimate of the form $\eta\geq \gamma_7 \kappa^{-\gamma_8}$ would be sufficient for us. We can use the celebrated result by Dobrowolski \cite{Dobro}, or a previous weaker result by Blanksby and Montgomery \cite{BlaMon}.

In case all the $\alpha_j$ are all torsion, it is clear that one can take $|\bo{b}_i|\leq \omega$ and use \eqref{tor}.
\end{proof}

\section{Bounded height}

In this section we see that the height of the points on the curve $\cC$ for which there is a dependence relation between the $P_i$ is bounded and a few consequences of this fact.

Let $k$ be a number field over which $\cC$ is defined. Suppose also that the finitely many points we excluded from $\cC$ to get $\widehat{\cC}$, which are algebraic, are defined over $k$.

Let $\cC'$ be the set of points $\bo{c}_0 \in \widehat{\cC}(\C)$ for which we have that $P_1(\bo{c}_0),\dots ,P_n(\bo{c}_0)$ satisfy a non-trivial relation on $E_{\lambda(\bo{c}_0)}$ and $Q_1(\bo{c}_0),\dots ,Q_{m}(\bo{c}_0)$ satisfy a non-trivial relation on $E_{\mu(\bo{c}_0)}$ (or $u_1(\bo{c}_0),\dots ,u_{m}(\bo{c}_0)$ are multiplicatively dependent).
Since  $\cC$ is defined over $\Qbar$,  the points in $\cC'$ must be algebraic. Moreover, by Silverman's Specialization Theorem \cite{Sil83}, there exists $\gamma_1>0$ such that 
\begin{equation}\label{boundedheight}
h(\bo{c}_0)\leq \gamma_1,
\end{equation}
for all $\bo{c}_0 \in \cC'$.

We see now a few consequences of this bound.
If $\delta>0$ is a small real number, let us call
\begin{equation*} 
\cC^{\delta}=\lg \bo{c} \in \cC: \Vert \bo{c} \Vert \leq \frac{1}{\delta}, \Vert \bo{c}-\bo{c}' \Vert \geq  \delta \mbox{ for all $\bo{c}' \in \cC\setminus \widehat{\cC}$} \rg .
\end{equation*}
Here $\Vert \cdot \Vert$ indicates the standard norm on $\C^{2n+2m+2}$ or $\C^{2n+m+1}$.

\begin{lemma}\label{lemconj}
There is a positive $\delta$ such that there are at least $\frac{1}{2}[k(\bo{c}_0):k]$ different $k$-embeddings $\sigma$ of $k(\bo{c}_0)$ in $\C$ such 
that $\sigma (\bo{c}_0)$ lies in $\cC^{\delta}$ for all $\bo{c}_0 \in \cC'$. 
\end{lemma}

\begin{proof}
See Lemma 8.2 of \cite{MZ14a}.
\end{proof}

\begin{remark}
We would like to point out that it might be possible to avoid the restriction to a compact domain and the use of the previous lemma by exploiting the work of Peterzil and Starchenko \cite{PetStar}, who proved that it is possible to define the Weierstrass $\wp$ function globally in the structure $\R_{\text{an,exp}}$.
\end{remark}

\begin{lemma}\label{lemboundheightdeg}
There exists a positive constant $\gamma_2$ such that, for every $\bo{c}_0 \in \cC'$, every $i=1, \dots , n$, and every $j=1,\dots , m$ we have
$$
\widehat{h}(P_i(\bo{c}_0)), \widehat{h}(Q_j(\bo{c}_0))\leq \gamma_2.
$$
\end{lemma}

\begin{proof}
We have $h(P_i(\bo{c}_0)) \leq h(\bo{c}_0)$ and, using the work of Zimmer \cite{Zimmer}, we have $\widehat{h}(P_i(\bo{c}_0))\leq h(P_j(\bo{c}_0)) +\gamma_3(h(\lambda(\bo{c}_0))+1)$. The same inequalities hold for the $Q_j$. The claim now follows from (\ref{boundedheight}).
\end{proof}

\section{Proof of Theorems \ref{mainthm} and \ref{mainthm2}}

Let us start with Theorem \ref{mainthm}.

By Northcott's Theorem \cite{Northcott1949} and the bound (\ref{boundedheight}) for the height, we only need to bound the degree of $\bo{c}_0$ over $k$, for all the $\bo{c}_0 \in \cC'$. 

Fix one $\bo{c}_0 \in \cC'$ and $d_0=[k(\bo{c}_0):k]$ which we suppose large. First, by Lemma \ref{lemconj}, we can choose $\delta$, independent of $\bo{c}_0 $, such that $\bo{c}_0$ has at least 
$\frac{1}{2}d_0$ conjugates in $\cC^{\delta}$.
Now, since $\cC^{\delta}$ is compact, there are $\bo{c}_1,\dots,  \bo{c}_{\gamma_2} \in \widehat{\cC}$ with corresponding neighborhoods $N_{\bo{c}_1}, \dots ,N_{\bo{c}_{\gamma_2}}$ and $D_{\bo{c}_1}, \dots ,D_{\bo{c}_{\gamma_2}} \subseteq \pi(\widehat{\cC})$, where $D_{\bo{c}_i} \subseteq \pi (N_{\bo{c}_i})$ contains $\pi(\bo{c}_i)$ and is homeomorphic to a closed disc and we have that the $\pi^{-1}(D_{\bo{c}_i}) \cap N_{\bo{c}_i}$ cover $\cC^{\delta}$.

We can suppose that $D_{\bo{c}_1}$ contains $t_0^\sigma=\pi(\bo{c}_0^\sigma)$ for at least $\frac{1}{2\gamma_2}d_0$ conjugates $\bo{c}_0^\sigma$. Since each $t \in \pi(\cC)$ has a uniformly bounded number of preimages $\bo{c} \in \cC$, we can suppose we have at least $\frac{1}{\gamma_3}d_0$ distinct such $t_0^\sigma$ in $D_{\bo{c}_1}$.

Now, the corresponding points $P_1(\bo{c}_0^\sigma), \dots ,P_n(\bo{c}_0^\sigma)$, $Q_1(\bo{c}_0^\sigma), \dots ,Q_m(\bo{c}_0^\sigma)$ 
satisfy the same relations. So there are $\bo{a}=(a_{1},\dots , a_{n})\in \Z^n\setminus \{0\}$ and $\bo{b}=(b_{1},\dots , b_{m}) \in R_2^m\setminus \{0\}$ such that
\begin{equation}\label{eq1}
\lg
\begin{array}{c}
a_1P_1(\bo{c}_0^\sigma)+\dots +a_nP_n(\bo{c}_0^\sigma)= O \text{ on } E_{\lambda(\bo{c}_0^\sigma)},\\ 
b_1Q_1(\bo{c}_0^\sigma)+\dots +b_m Q_m(\bo{c}_0^\sigma)= O \text{ on } E_{\mu(\bo{c}_0^\sigma)}.
\end{array}
\right. 
\end{equation}

By Lemma \ref{lemboundheightdeg}, $\widehat{h}(P_i(\bo{c}_0^\sigma)),\widehat{h}(Q_j(\bo{c}_0^\sigma))\leq \gamma_4$. Moreover, the $P_i(\bo{c}_0^\sigma)$ and $Q_j(\bo{c}_0^\sigma)$ are defined over a number field $K$ of degree $\ll d_0$ over $\Q$. Therefore, applying Lemma \ref{lemgeneratorsell} and recalling \eqref{boundedheight}, we can suppose that 
\begin{equation}\label{smallgen}
|\bo{a}|,|\bo{b}|\leq\gamma_{5} d_0^{\gamma_{6}}.
\end{equation}
Recall that, in case $\Z \subsetneq R_2 = \Z+\rho \Z$, we set $|\bo{b}|=\max \{|N(b_1)|, \dots , |N( b_{m})|\}$ and we can just apply Lemma \ref{lemgeneratorsell} to $Q_1, \dots ,Q_m, \rho Q_1, \dots , \rho Q_m$ noting that $\widehat{h}(\rho Q_j)\ll \widehat{h}(Q_j)$.

Now, recall that, in Section \ref{periods}, on $ D_{\bo{c}_1}$ we defined $f,g$ to be generators of the period lattice of $E_\lambda$ and the elliptic logarithms $z_1, \dots ,z_n$ such that, if $\bo{c}$ is the only point in $N_{\bo{c}_1}$ above $t$,
\begin{equation*}\label{elllog}
\exp_\lambda (z_i(t))=P_i(\bo{c}),
\end{equation*}
on $D_{\bo{c}_1}$ and $h,k,w_1, \dots ,w_m$ as generators for the period lattice and elliptic logarithms of the $Q_j$ for $E_{\mu}$.

By \eqref{eq1}, we have that
\begin{equation*}
\lg
\begin{array}{c}
a_1z_1(t_0^\sigma)+\dots +a_nz_n(t_0^\sigma) \in \Z f(t_0^\sigma) + \Z g(t_0^\sigma),  \\
b_1w_1(t_0^\sigma)+\dots +b_m w_m(t_0^\sigma) \in \Z h(t_0^\sigma) + \Z k(t_0^\sigma).
\end{array}
\right.
\end{equation*}

Recall the definition of $D_{\bo{c}_1}(\bo{a},\bo{b})$ in \eqref{defD}.
By Proposition \ref{estim} and \eqref{smallgen}, we have that $|D_{\bo{c}_1}(\bo{a},\bo{b})|\ll_\epsilon d_0^{\gamma_6 \epsilon}$.
But by our choice of $D_{\bo{c}_1}$ we had at least $\frac{1}{\gamma_3}d_0$ points in $D_{\bo{c}_1}(\bo{a},\bo{b})$. Therefore, if we choose $\epsilon = \frac{1}{2 \gamma_6}$ we have a contradiction if $d_0$ is large enough.

We have just deduced that $d_0$ is bounded and, by (\ref{boundedheight}) and Northcott's Theorem, we have the claim of Theorem \ref{mainthm}.

Theorem \ref{mainthm2} can  be proved following the same line and combining Lemma \ref{lemgeneratorsell} with Lemma \ref{lemgeneratorsmult}.

\section{Proof of Theorem \ref{thmscheme}}\label{proofsch}

We recall our setting.
We have $\mathcal{E}_i \rightarrow S$ non-isotrivial elliptic schemes and, for all $i$, we let $\mc{A}_i$ be the $n_i$-fold fibered power of $\cE_i$ over $S$.
Moreover, we have $E_1, \dots ,E_p$ elliptic curves pairwise non-isogenous, which we consider as constant families over $S$.
We defined $\mc{A}$ to be the fibered product
$$
\mc{A}_1\times_S \dots \times_S \mc{A}_l \times_S E_1^{m_1}  \times_S\dots \times_S E_{p}^{m_p}\times_S \mathbb{G}_m^q
$$
 over $S$. We suppose that everything is defined over a number field $k$.

Fix $i_0$, with $1\leq i_0 \leq l$. For every $\bo{a}=(a_1, \dots , a_{n_{i_0}})\in \Z^{n_{i_0}} $ we have a morphism $\bo{a}:\cA_{i_0} \rightarrow \cE_{i_0}$ defined by
$$
\bo{a}(P_1,\dots, P_{n_{i_0}})=a_1 P_1 +\dots +a_{n_{i_0}} P_{n_{i_0}}.
$$
We identify the elements of $\Z^{n_{i_0}}$ with the morphisms they define. The fibered product $\bo{a}_1 \times_S \dots \times_S \bo{a}_r$, for $\bo{a}_1 , \dots , \bo{a}_r \in \Z^{n_{i_0}}$ defines a morphism $\cA_{{i_0}} \rightarrow \mathcal{B}$ over $S$ where $\mathcal{B}$ is the $r$-fold fibered power of $\cE_{{i_0}}$. Similarly, for $j_0$, $1\leq j_0 \leq p$, vectors $\bo{b}\in R_{j_0}^{m_{j_0}}$, where $R_{j_0}$ is the endomorphism ring of $E_{j_0}$, define morphisms from $E_{j_0}^{m_0}$ to $E_{j_0}$ and vectors $\bo{c}\in \Z^{q}$ define morphisms from $\mathbb{G}_m^q$ to $\mathbb{G}_m$. Therefore, square matrices with entries in $\Z$ or in an eventually larger $R_{j_0}$ and appropriate size will define endomorphisms of $\cA_{i_0}$, of $E_{j_0}^{m_{j_0}}$ or of $\mathbb{G}_m^q$.
Finally, we can take the fibered product of such endomorphisms to obtain an endomorphism of $\cA$ which will be represented by a block diagonal matrix whose blocks correspond to the endomorphisms defined above. These matrices form a ring which we call $R$.

For an $\alpha\in R$, the kernel of $\alpha$, $\text{ker} \, \alpha$ indicates the fibered product of $\alpha: \cA \rightarrow \cA$ with the zero section $S\rightarrow \cA$. We consider it as a closed subscheme of $\cA$.

\begin{lemma}\label{lemmat}
Let $G$ be a flat subgroup scheme of $\mc{A}$ of codimension $\geq d$; then, there exists an $\alpha \in R$ of rank $d$ such that $G\subseteq \ker \alpha$ and, for any $\alpha$ of rank $d$, $\ker \alpha$ is a flat subgroup scheme of codimension $d$.
\end{lemma}

\begin{proof}
The lemma can be proved following the line of the proof of Lemma 2.5 of \cite{Hab}. The fact that there is an $s\in S(\C)$ whose endomorphism ring is exactly $R$ follows from Corollary 1.5 of \cite{Noot}.
\end{proof}

From this lemma, we can deduce that each flat subgroup scheme of $\cA$ is contained in a flat subgroup scheme of the same dimension and of the form
$$
G_1\times_S \dots \times_S G_l \times_S H_1  \times_S\dots \times_S H_p \times_S T
$$
where $G_i$ is a flat subgroup scheme of $\cA_i$, $H_j$ an algebraic subgroup of $E_j^{m_j}$ and $T$ an algebraic subgroup of $\mathbb{G}_m^q$.
It is then clear that, by projecting and recalling that we suppose $p$ or $q$ equal to 0, we only need to prove our Theorem \ref{thmscheme} for the following cases:
\begin{enumerate}
\item $\cA=\cA_1 \times_S \cA_2$;
\item $\cA=\cA_1\times_S E_1^{m_1}$;
\item $\cA=\cA_1\times_S \mathbb{G}_m^q$;
\item $\cA=E_1^{m_1} \times_S E_2^{m_2}$.
\end{enumerate}

Moreover, in case (4) the Theorem follows from Theorem 1.1 of Habegger and Pila \cite{HabPila14}, so we are left with the first three cases.\\

We first consider case (1). We need to perform a base change to the Legendre family. 
Consider now the Legendre family with equation \eqref{legendre}. This gives an example of an elliptic scheme, which we call $\cE_L$, over the modular curve $Y(2)=\mathbb{P}^1\setminus \{0,1,\infty\}$. We write $\cE_L^{(g)}$ for the $g$-fold fibered power of $\cE_L$.
We call $\pi_i$ (resp.~ $\pi_L^{(g)}$) the structural morphism $\cA_i\rightarrow S$ (resp.~ $\cE_L^{(g)}\rightarrow Y(2)$).

\begin{lemma} \label{lemdiag}
Let $\cA=\cA_1 \times_S \cA_2$.
After possibly replacing $S$ by a Zariski open, non-empty subset there exist irreducible, non-singular quasi-projective curves $S'$ and $S''\subseteq Y(2)\times Y(2)$ defined over $\Qbar$ such that the following is a commutative diagram
\begin{equation}\label{diag}
\begin{CD}
\cA @<{f}<< \cA' @>e>>  \cA''\\
@V{\pi_1 \times_S \pi_2}VV @VVV @VV{\pi_L^{(n_1)}\times \pi_L^{(n_2)}}V\\
S @<<{l}<  S' @>>{\lambda}> S''
\end{CD}
\end{equation}
where $l$ is finite, $\lambda$ is quasi-finite, $\cA'$ is the abelian scheme $(\cA_1\times_S \cA_2)\times_S S'$, $\cA''=\cE_L^{(n_1)} \times \cE_L^{(n_2)}$, $f$ is finite and flat and $e$ is quasi-finite and flat. Moreover, the restriction of $f$ and $e$ to any fiber of $\cA'\rightarrow S'$ is an isomorphism of abelian varieties.
\end{lemma}

\begin{proof}
We follow the line of Lemma 5.4 of \cite{Hab} and skip several details which can be found there. We fix an extension $K$ of $k(S)$ such that, for $i=1,2$, the generic fiber of $\cE_i\rightarrow S$ is isomorphic to an elliptic curve of equation $y^2=x(x-1)(x-\lambda_i)$, for $\lambda_i \in K$. The field $K$ is the function field of an irreducible, non-singular projective curve $\overline{S}'$ and we have a finite morphism $\overline{S}'\rightarrow \overline{S}$. We let $S'$ be the preimage of $S$ in $\overline{S}'$ and call $l:S'\rightarrow S$ the restriction of the above morphism which remains finite. Moreover, we have finite morphisms $\lambda_i:S' \rightarrow Y(2)$. We may shrink $S$ and suppose that the $\lambda_i$ and $l$ are \'etale. 
For $i=1,2$, by Lemma 5.4 of \cite{Hab} we have the commutative diagrams
$$
\begin{CD}
\cA_i @<{f_i}<< \cA'_i @>e_i>>  \cE_L^{(n_i)}\\
@V{\pi_i}VV @VVV @VV{\pi_L^{(n_i)}}V\\
S @<<{l}<  S' @>>{\lambda_i}> Y(2)
\end{CD}
$$
where $\cA'_i=\cA_i\times_S S'$, $f_i$ is finite and flat and $e_i$ is quasi-finite and flat.

The square on the left of \eqref{diag} is the appropriate fibered product over $S$ of the morphisms $\cA'\rightarrow \cA'_i\rightarrow \cA_i$. 

Now, we have the diagram
$$
\begin{CD}
 \cA' @>e_1\times e_2>>  \cE_L^{(n_1)} \times \cE_L^{(n_2)}\\
 @VVV @VV{\pi_L^{(n_1)}\times \pi_L^{(n_2)}}V\\
 S' @>>{\lambda_1\times \lambda_2}> Y(2)\times Y(2)
\end{CD}
$$
We set $S''=\lambda(S')=(\lambda_1, \lambda_2)(S')$, restrict the base of the abelian scheme $\cE_L^{(n_1)} \times \cE_L^{(n_2)}\rightarrow Y(2)\times Y(2)$ to $S''$ and call $e$ the resulting map $\cA' \rightarrow \cE_L^{(n_1)} \times \cE_L^{(n_2)}$. We then have the square on the right of \eqref{diag}.
Finally, to prove the claimed properties of $f$ and $e$ and the fact that they are isomorphisms of abelian varieties when restricted to the fibers one can proceed as in Lemma 5.4 of \cite{Hab}.
\end{proof}

One can prove that analogous results hold for cases (2) and (3).\\

We will also need the following technical lemma which holds in all three cases.

\begin{lemma} \label{lemhab}
If $G$ is a flat subgroup scheme of $\cA$ then $e\left(f^{-1}(G)\right)$ is a flat subgroup scheme of $\cA''$ of the same dimension.
Moreover, let $X$ be a subvariety of $\cA$ dominating $S$ and not contained is a proper flat subgroup scheme of $\cA$, $X''$ an irreducible component of $f^{-1}(X)$ and $X'$ the Zariski closure of $e(X'')$ in $\cA''$. Then $X'$ has the same dimension of $X$, dominates $S''$ and is not contained in a proper flat subgroup scheme of $\cA''$.
\end{lemma}

\begin{proof}
This follows from the proof of Lemma 5.5 of \cite{Hab}.
\end{proof}

We are now ready to see how Theorem \ref{thmscheme} can be deduced from Theorems \ref{mainthm} and \ref{mainthm2} and earlier results.

Consider case (1).
We can assume that $\cC$ is not contained in a proper flat subgroup scheme of $\cA$. Therefore, it is enough to show that $\cC\cap \bigcup G$ is finite where the union is taken over all flat subgroup schemes of codimension at least 2. 

Take $\cC'$ an irreducible component of $f^{-1}\left(\cC \right) $ and consider the Zariski closure $\cC''$ of $e(\cC')$. By Lemma \ref{lemhab}, $\cC''$ is a curve in $\cA''$ dominating $S''$ and not contained in a proper flat subgroup scheme.

Now, since $e$ is quasi-finite, if $e\left(f^{-1}\left(\cC\cap \cA^{\{2\}}\right)\right)$ is finite then $\cC\cap \cA^{\{2\}}$ is finite and, by Lemma \ref{lemhab}, we have
$$
e\left(f^{-1}\left(\cC\cap \cA^{\{2\}}\right)\right) \subseteq e\left(f^{-1}\left(\cC \right) \right) \cap \cA''^{\{2\}}.
$$
Therefore, we can reduce to proving our claim for $\cA''$ and for $\cC''$. 
Note that the generic fiber of $\cA''$ is isomorphic to the generic fiber of $\cA$ and is therefore a product of powers of non-isogenous elliptic curves. 

By Lemma \ref{lemmat}, each flat subgroup scheme of codimension at least 2 of $\cA''$ is contained in $\textnormal{ker}\, \alpha$ for some $\alpha$ of rank 2, where $\alpha$ is a block diagonal matrix with two blocks of respective sizes $n_1$ and $n_2$ and has two non-zero rows. In case these two rows are in the same block, then we are in the case of the Theorem 2.1 of \cite{linrel} while, if they are in two different blocks, then we are in the case of Theorem \ref{mainthm}. In any case, $\cC''$ intersects only finitely many such flat subgroup schemes and we have the claim in case (1).\\

For the other two cases one proceeds in the same way. If the two non-zero rows of $\alpha$ are contained in two different blocks then we apply Theorem \ref{mainthm} or \ref{mainthm2}. If they are in the same block then one can use results of Viada \cite{Viada2008} and Galateau \cite{galateau2010} for case (2) and Maurin \cite{Maurin} for case (3). This concludes the proof of Theorem \ref{thmscheme}.

\section*{Funding}

This work was supported by the European Research Council [267273], the Engineering and Physical Sciences Research Council [EP/N007956/1 to F.B. and EP/N008359/1 to L.C.], the Istituto Nazionale di Alta Matematica [Borsa Ing. G. Schirillo to L.C.] and the Swiss National Science Foundation [165525 to F.B.].

\section*{Acknowledgments}

The authors would like to thank Umberto Zannier for his support and Daniel Bertrand, Gareth Jones, Lars K\"uhne and Harry Schmidt for useful discussions.

\bibliographystyle{plain}
\bibliography{bibliography}

\end{document}